\documentclass[12pt]{amsproc}
\usepackage{amsfonts}

\theoremstyle{plain}

\newtheorem{corollary}{Corollary}

\newtheorem{definition}{Definition}

\newtheorem{proposition}{Proposition}
\newtheorem{remark}{Remark}

\newtheorem{theorem}{Theorem}
\numberwithin{equation}{section}
\input{tcilatex}

\begin{document}
\title[Some properies of F-harmonic maps]{Some properties of F-harmonic maps}
\author{Mohammed Benalili}
\address{Dept. of Mathematics, Facult\'{e} des Sciences, Universit\'{e}
Abou-Belka\"{\i}d, Tlemcen}
\email{m\_benalili@mail.univ-tlemcen.dz}
\author{Hafida Benallal}
\email{h\_benallal@mail.univ-tlemcen.dz}

\begin{abstract}
In this note, we investigate estimates of the Morse index for $F$-harmonic
maps into spheres, our results extend partially those obtained in (\cite{14}%
) and (\cite{15}) for harmonic and $p$-harmonic maps.
\end{abstract}

\maketitle

\bigskip

\section{\protect\bigskip Introduction}

Harmonic maps have been studied first by J. Eells and J.H.Sampson in the
sixties and since then many works were done ( see \cite{4}, \cite{9}, \cite%
{13}, \cite{16}, \cite{17}, \cite{21}) to cite a few of them. Extensions to
notions of $p$-harmonic, biharmonic, $F$-harmonic and $f$-harmonic maps were
introduced and similar research has been carried out (see \cite{1}, \cite{2}%
, \cite{3}, \cite{5}, \cite{12}, \cite{15}, \cite{18}, \cite{20}). Harmonic
maps were applied to broad areas in sciences and engineering including the
robot mechanics ( see \cite{6}, \cite{8} )$.$

The Morse index for harmonic maps, $p$-harmonic maps, as well as biharmonic
maps, into a standard unit Euclidean sphere $S^{n}$ has been widely
considered ( see \cite{12}, \cite{14}, \cite{15},).

In this paper for a $C^{2}$-function $F:\left[ 0,+\infty \right[ \rightarrow %
\left[ 0,+\infty \right[ $ such that $F^{\prime }(t)>0$ on $t\in \left]
0,+\infty \right[ $, we consider the Morse index for $F$-harmonic maps into
spheres. Our results generalize partial estimates of the Morse index
obtained in (\cite{14}) and (\cite{15}) for harmonic and $p$-harmonic maps.

Let $(M,g)$ be a compact Riemannian manifold of dimension $m\geq 2$, $S^{n%
\text{ }}$the unit $n$-dimensional Euclidean sphere with $n\geq 2$ endowed
with the canonical metric $can$ induced by the inner product of $R^{n+1}$.

For a $C^{1}$- application $\phi :(M,g)\longrightarrow (S^{n},can)$, we
define the $F$-energy functional by, 
\begin{equation*}
E_{F}(\phi )=\int_{M}F\left( \frac{\left\vert d\phi \right\vert ^{2}}{2}%
\right) dv_{g}
\end{equation*}%
where $\frac{\left\vert d\phi \right\vert ^{2}}{2}$ denotes the energy
density given by 
\begin{equation*}
\frac{\left\vert d\phi \right\vert ^{2}}{2}=\frac{1}{2}\sum_{i=1}^{m}\left%
\vert d\phi (e_{i})\right\vert ^{2}
\end{equation*}%
and where $\left\{ e_{i}\right\} $ is an orthonormal basis on $T_{x}$ $M$
and $dv_{g}$ is the Riemannian measure associated to $g$ on $M$.

Let $\phi ^{-1}TS^{n}$ be the pullback vector fiber bundle of $TS^{n}$, $%
\Gamma \left( \phi ^{-1}TS^{n}\right) $ the space of sections on $\phi
^{-1}TS^{n}$ and denote by $\nabla ^{M}$, $\nabla ^{S^{n}}$and $\tilde{\nabla%
}$ Levi-Civita connections on $TM$, $T$ $S^{n}$ and $\phi ^{-1}TS^{n}$
respectively. $\tilde{\nabla}$ is defined by \ 
\begin{equation*}
\tilde{\nabla}_{X}Y=\nabla _{\phi _{\ast }X}^{S^{n}}Y
\end{equation*}%
where $X\in TM$ and $Y\in \Gamma \left( \phi ^{-1}TS^{n}\right) $.

Let $v$ be a vector field on $S^{n}$ and $\left( \phi _{t}^{v}\right) _{t}$
the flow of diffeomorphisms induced by $v$ on $S^{n}$ i.e. 
\begin{equation*}
\phi _{0}^{v}=\phi \ \ \text{, \ \ }\frac{d}{dt}\phi _{t}^{v}\mid _{t=0}=v%
\text{.}
\end{equation*}%
The first variation formula of $E_{F}(\phi )$ is given by%
\begin{equation*}
\frac{d}{dt}E_{F}(\phi _{t})\mid _{t=0}=\int_{M}F^{\prime }\left( \frac{%
\left\vert d\phi _{t}\right\vert ^{2}}{2}\right) \left\langle \nabla
_{\partial t}d\phi _{t},d\phi _{t}\right\rangle \left\vert _{t=0}\right.
dv_{g}
\end{equation*}%
\begin{equation*}
=-\int_{M}\left\langle v,\tau _{F}(\phi )\right\rangle dv_{g}
\end{equation*}%
where $\tau _{F}(\phi )=trace_{g}\nabla \left( F^{\prime }\left( \frac{%
\left\vert d\phi \right\vert ^{2}}{2}\right) d\phi \right) $ denotes the
Euler-Lagrange equation of the $F$-energy functional $E_{F}$. Remark that if 
$\left\vert d\phi \right\vert _{\phi ^{-1}TN}$ is constant then $\phi $ is
harmonic if and only if $\phi $ is $F$-harmonic.

\begin{definition}
$\phi $ is called $F$-harmonic if and only if $\tau _{F}(\phi )=0$ i.e. $%
\phi $ is a critical point of $\ F$-energy functional $E_{F}$.
\end{definition}

The second variation of $E_{F}$ is given as 
\begin{equation*}
\frac{d^{2}}{dt^{2}}E_{F}(\phi _{t})\mid _{t=0}=\frac{d}{dt}\int_{M}\frac{d}{%
dt}F\left( \frac{\left\vert d\phi _{t}\right\vert ^{2}}{2}\right) \left\vert
_{t=0}\right. dv_{g}
\end{equation*}%
\begin{equation*}
=\int_{M}\left[ F^{\prime \prime }\left( \frac{\left\vert d\phi \right\vert
^{2}}{2}\right) \left\langle \nabla v,d\phi _{t}\right\rangle ^{2}+F^{\prime
}\left( \frac{\left\vert d\phi \right\vert ^{2}}{2}\right) \left\vert \nabla
v\right\vert ^{2}\right] dv_{g}
\end{equation*}%
\begin{equation*}
-\int_{M}\left\langle \nabla _{\partial t}\frac{\partial \phi _{t}}{\partial
t}\left\vert _{t=0}\right. ,trace_{g}\nabla \left( F^{\prime }\left( \frac{%
\left\vert d\phi \right\vert ^{2}}{2}\right) d\phi \right) \right\rangle
dv_{g}
\end{equation*}%
\begin{equation*}
-\int_{M}F^{\prime }\left( \frac{\left\vert d\phi \right\vert ^{2}}{2}%
\right) \sum_{i=1}^{m}\left\langle R^{S^{n}}\left( v,d\phi (e_{i})\right)
d\phi (e_{i}),v\right\rangle dv_{g}
\end{equation*}%
and since $\phi $ is $F$-harmonic, $\tau _{F}(\phi )=0$, then%
\begin{equation*}
\frac{d^{2}}{dt^{2}}E_{F}(\phi _{t})\mid _{t=0}=\int_{M}F^{\prime \prime
}\left( \frac{\left\vert d\phi \right\vert ^{2}}{2}\right) \left\langle
\nabla v,d\phi \right\rangle ^{2}dv_{g}+
\end{equation*}%
\begin{equation}
\int_{M}F^{\prime }\left( \frac{\left\vert d\phi \right\vert ^{2}}{2}\right) %
\left[ \left\vert \nabla v\right\vert ^{2}-\sum_{i=1}^{m}\left\langle
R^{S^{n}}\left( v,d\phi (e_{i})\right) d\phi (e_{i}),v\right\rangle \right]
dv_{g}\text{.}  \label{4}
\end{equation}

Along this paper we consider variation in directions of vector fields of the
subspace $\pounds (\phi )$ of $\Gamma (\phi ^{-1}TS^{n})$ defined by 
\begin{equation*}
\pounds (\phi )=\left\{ \bar{v}\circ \phi ,v\in 
\mathbb{R}
^{n+1}\right\}
\end{equation*}%
where $\bar{v}$ is a vector field on $S^{n}$ given by $\bar{v}%
(y)=v-\left\langle v,y\right\rangle y$ for any $y\in S^{n}$; it is known
that $\bar{v}$ is a conformal vector field on $S^{n}$. Obviously, if $\phi $
is not constant, $\pounds (\phi )$ is of dimension $n+1$.

\section{Morse index for $F$-harmonic application}

For any vector field $v$ on $S^{n}$ along $\phi $, we associate the
quadratic form 
\begin{equation*}
Q_{\phi }^{F}(v)=\frac{d^{2}}{dt^{2}}E_{F}(\phi _{t})\mid _{t=0}\text{.}
\end{equation*}

The Morse index of the $F$-harmonic map is defined as the positive integer

\begin{equation*}
Ind_{F}(\phi )=\sup \left\{ \dim N,N\subset \Gamma (\phi )\text{ such that }%
Q_{\phi }^{F}\left( v\right) \text{ negative defined on }N\right\}
\end{equation*}%
where $N$ is a subspace of $\Gamma (\phi )$. The Morse index measures the
degree of the instability of $\phi $ which is called $F$- stable if $%
Ind_{F}(\phi )=0$. Let also $S_{g}^{F}(\phi )$ be the $F$-stress-energy
tensor defined by%
\begin{equation}
S_{g}^{F}(\phi )=F^{\prime }(\frac{\left\vert d\phi \right\vert ^{2}}{2}%
)\left\vert d\phi \right\vert ^{2}g-2\left( F^{\prime }(\frac{\left\vert
d\phi \right\vert ^{2}}{2})+F^{\prime \prime }\left( \frac{\left\vert d\phi
\right\vert ^{2}}{2}\right) \frac{\left\vert d\phi \right\vert ^{2}}{2}%
\right) \phi ^{\ast }can\text{.}  \label{4'}
\end{equation}%
For $x\in M$, we put 
\begin{equation*}
S_{g}^{o,F}(\phi )=\inf \left\{ S_{g}^{F}(\phi )(X,X)\text{, }X\in T_{x}M%
\text{ such that }g(X,X)=1\right\} \text{.}
\end{equation*}%
The tensor $S_{g}^{F}(\phi )$ will be called positive ( resp. positive
defined) at $x$ if $\ S_{g}^{o,F}(\phi )\geq 0$ (resp. $S_{g}^{o,F}(\phi )>0$
).

\begin{remark}
$F(t)=\frac{1}{p}\left( 2t\right) ^{\frac{p}{2}}$, with $p\in \left[
2,+\infty \right[ $, $S_{g}^{p}\left( \phi \right) $ is the stress-energy
tensor introduced by Eells and Lemaire for $p=2$ ( \cite{9})or El Soufi for $%
p\geq 4$, (\cite{13}).
\end{remark}

In this note we state the following result

\begin{theorem}
\label{theo} Let $\phi $ be an $F$-harmonic map from a compact $m-$%
Riemannian manifold $\left( M,g\right) $ ($m\geq 2$) into the Euclidean
sphere $S^{n}$ ($n\geq 2$). Suppose that the $F$-stress-energy tensor $%
S_{g}^{F}\left( \phi \right) $ of $\phi $ is positive defined. Then the
Morse index of $\phi $, $Ind_{F}(\phi )\geq n+1$.
\end{theorem}

\begin{proof}
Let $w=\bar{v}\circ \phi \in \pounds (\phi )$ and put $\left\langle v,\phi
\right\rangle =\phi _{v}$. For any point $x\in M$, we denote respectively by 
$w^{T}(x)$ and $w^{\perp }(x)$ the tangential and normal components of the
vector $w(x)$ on the spaces $d\phi (T_{x}M)$ and $d\phi (T_{x}M)^{\perp }$.
Let also $\left\{ e_{1},...,e_{m}\right\} $ an orthonormal basis of $T_{x}M$
which diagonalizes $\phi ^{\ast }can$ and such that $\left\{ d\phi
(e_{1}),d\phi (e_{2}),...,d\phi (e_{l})\right\} $ forms a basis of $d\phi
(T_{x}M)$.

If $\left( F^{\prime }\left( \frac{\left\vert d\phi \right\vert ^{2}}{2}%
\right) +\frac{\left\vert d\phi \right\vert ^{2}}{2}F^{\prime \prime }\left( 
\frac{\left\vert d\phi \right\vert ^{2}}{2}\right) \right) \neq 0$ at the
point $x$, then%
\begin{equation*}
\left\vert \overline{v}^{T}(x)\right\vert ^{2}=\sum_{i=1}^{l}\left\vert
d\phi (e_{i})\right\vert ^{-2}\left\langle \overline{v}(x),d\phi
(e_{i})\right\rangle ^{2}
\end{equation*}%
on the other hand, for any $i\leq l$, we have%
\begin{equation*}
2\left( F^{\prime }\left( \frac{\left\vert d\phi \right\vert ^{2}}{2}\right)
+\frac{\left\vert d\phi \right\vert ^{2}}{2}F^{\prime \prime }\left( \frac{%
\left\vert d\phi \right\vert ^{2}}{2}\right) \right) \left\vert d\phi
(e_{i})\right\vert ^{2}=\left\vert d\phi \right\vert ^{2}F^{\prime }(\frac{%
\left\vert d\phi (x)\right\vert ^{2}}{2})
\end{equation*}%
\begin{equation}
-S_{g}^{F}(\phi )(x)(e_{i},e_{i})\leq \left\vert d\phi \right\vert
^{2}F^{\prime }(\frac{\left\vert d\phi (x)\right\vert ^{2}}{2}%
)-S_{g}^{o,F}(\phi )(x)  \label{5}
\end{equation}%
so 
\begin{equation*}
\left( \left\vert d\phi \right\vert ^{2}F^{\prime }(\frac{\left\vert d\phi
(x)\right\vert ^{2}}{2})-S_{g}^{o,F}(\phi )(x)\right) \left\vert \overline{v}%
^{T}(x)\right\vert ^{2}\geq 
\end{equation*}%
\begin{equation*}
2\left( F^{\prime }\left( \frac{\left\vert d\phi \right\vert ^{2}}{2}\right)
+\frac{\left\vert d\phi \right\vert ^{2}}{2}F^{\prime \prime }\left( \frac{%
\left\vert d\phi \right\vert ^{2}}{2}\right) \right)
\sum_{i=1}^{l}\left\langle \overline{v}(x),d\phi (e_{i})\right\rangle ^{2}
\end{equation*}%
and since,%
\begin{equation*}
\left\langle \overline{v}(x),d\phi (e_{i})\right\rangle ^{2}=\left\langle
v-\left\langle v,\phi \right\rangle \phi ,d\phi (e_{i})\right\rangle ^{2}
\end{equation*}%
\begin{equation*}
=\left\langle v,d\phi (e_{i})\right\rangle ^{2}=\left\vert d\phi
_{v}(e_{i})\right\vert ^{2}
\end{equation*}%
we get 
\begin{equation*}
\left( \left\vert d\phi \right\vert ^{2}F^{\prime }(\frac{\left\vert d\phi
(x)\right\vert ^{2}}{2})-S_{g}^{o,F}(\phi )(x)\right) \left\vert
w^{T}(x)\right\vert ^{2}\geq 2\left( F^{\prime }\left( \frac{\left\vert
d\phi \right\vert ^{2}}{2}\right) +\frac{\left\vert d\phi \right\vert ^{2}}{2%
}F^{\prime \prime }\left( \frac{\left\vert d\phi \right\vert ^{2}}{2}\right)
\right) \left\vert d\phi _{v}(x)\right\vert ^{2}\text{.}
\end{equation*}%
Now, taking into account (\ref{5}), we infer that%
\begin{equation*}
2\left( F^{\prime }\left( \frac{\left\vert d\phi \right\vert ^{2}}{2}\right)
+\frac{\left\vert d\phi \right\vert ^{2}}{2}F^{\prime \prime }\left( \frac{%
\left\vert d\phi \right\vert ^{2}}{2}\right) \right) \left\vert d\phi
_{v}(x)\right\vert ^{2}-\left\vert d\phi \right\vert ^{2}F^{\prime }(\frac{%
\left\vert d\phi (x)\right\vert ^{2}}{2})\left\vert \overline{v}\right\vert
^{2}
\end{equation*}%
\begin{equation*}
\leq -\left\vert d\phi \right\vert ^{2}F^{\prime }(\frac{\left\vert d\phi
(x)\right\vert ^{2}}{2})\left\vert \overline{v}^{N}\left( x\right)
\right\vert -S_{g}^{o,F}(\phi )(x)\left\vert \overline{v}^{T}(x)\right\vert
^{2}
\end{equation*}%
\begin{equation}
\leq -S_{g}^{o,F}(\phi )(x)\left\vert \overline{v}(x)\right\vert ^{2}
\label{6}
\end{equation}%
Now the second variation writes as 
\begin{equation*}
\frac{d^{2}}{dt^{2}}E_{F}(\phi _{t})\mid _{t=0}=\int_{M}F^{\prime \prime
}\left( \frac{\left\vert d\phi \right\vert ^{2}}{2}\right) \left\langle
\nabla \overline{v},d\phi \right\rangle ^{2}dv_{g}
\end{equation*}%
\begin{equation*}
+\int_{M}F^{\prime }\left( \frac{\left\vert d\phi \right\vert ^{2}}{2}%
\right) \left[ \left\vert \nabla \overline{v}\right\vert ^{2}-\left\vert
d\phi \right\vert ^{2}\left\vert \overline{v}\right\vert ^{2}+\left\vert
d\phi _{v}\right\vert ^{2}\right] dv_{g}
\end{equation*}%
Consequently, we have%
\begin{equation*}
Q_{\phi }^{F}(v)=2\int_{M}\left( \frac{\left\vert d\phi \right\vert ^{2}}{2}%
F^{\prime \prime }\left( \frac{\left\vert d\phi \right\vert ^{2}}{2}\right)
+F^{\prime }\left( \frac{\left\vert d\phi \right\vert ^{2}}{2}\right)
\right) \left\vert d\phi _{v}\right\vert ^{2}dv_{g}
\end{equation*}%
\begin{equation*}
-\int_{M}F^{\prime }\left( \frac{\left\vert d\phi \right\vert ^{2}}{2}%
\right) \left\vert d\phi \right\vert ^{2}\left\vert \overline{v}o\phi
\right\vert ^{2}dv_{g}
\end{equation*}%
and taking account of the inequality (\ref{6}), we get that 
\begin{equation*}
Q_{\phi }^{F}(v)\leq -2\int_{M}S_{g}^{0,F}(\phi )\left\vert \overline{v}%
\right\vert ^{2}dv_{g}\text{.}
\end{equation*}%
Finally since $S_{g}^{0,F}(\phi )$ is positive defined, it follows that $%
Q_{F}$ is negative defined on $\pounds (\phi )$. Hence%
\begin{equation*}
Ind_{_{F}}(\phi )\geq n+1\text{.}
\end{equation*}
\end{proof}

\section{Morse index of particular $F$-harmonic maps}

\subsection{Stability of the identity map}

In this section we borrow ideas from \cite{12} to show the stability of the
identity map. Let $\left( M,g\right) $ be a compact manifold and consider
the identity $I$ on $M$ which is obviously $F$-harmonic, the second
variation formula of $I$ writes as%
\begin{equation*}
Q_{I}^{F}\left( v\right) =F^{\prime \prime }\left( \frac{m}{2}\right)
\sum_{i=1}^{m}\int_{M}\left\langle \nabla _{e_{i}}v,e_{i}\right\rangle
^{2}dv_{g}+
\end{equation*}%
\begin{equation}
F^{\prime }\left( \frac{m}{2}\right) \int_{M}\left[ \left\vert \nabla
v\right\vert ^{2}-Ric_{M}\left( v,v\right) \right] dv_{g}\text{.}  \label{7}
\end{equation}%
If $L_{v}$ denotes the Lie derivative in the direction of $v$, the Yano's
formula \cite{22} leads to%
\begin{equation}
\int_{M}\left[ \left\vert \nabla v\right\vert ^{2}-Ric_{M}\left( v,v\right) %
\right] dv_{g}=\int_{M}\left[ \frac{1}{2}\left\vert L_{v}g\right\vert
^{2}-\left( div\left( v\right) \right) ^{2}\right] dv_{g}\text{.}  \label{8}
\end{equation}%
Now if $\left( e_{i}\right) _{i}$ is an orthonormal basis on $M$ which
diagonalizes $L_{v}g$ we obtain as in \cite{12} that 
\begin{equation}
\left\vert L_{v}g\right\vert ^{2}\geq \frac{4}{m}\left( div(v)\right) ^{2}
\label{9}
\end{equation}%
therefore by (\ref{7}), (\ref{8}) and (\ref{9}) we infer that%
\begin{equation}
Q_{I}^{F}\left( v\right) \geq \frac{1}{m}\left( F^{\prime \prime }\left( 
\frac{m}{2}\right) +\left( 2-m\right) F^{\prime }\left( \frac{m}{2}\right)
\right) \int_{M}div(v)^{2}dv_{g}\text{.}  \label{10}
\end{equation}%
We deduce the following proposition:

\begin{proposition}
Let $\left( M,g\right) $ be a compact Riemannian manifold of dimension $%
m\geq 3$. Suppose that%
\begin{equation}
F^{\prime \prime }\left( \frac{m}{2}\right) +\left( 2-m\right) F^{\prime
}\left( \frac{m}{2}\right) \geq 0\text{.}  \label{11}
\end{equation}%
The identity map $I$ on $M$ is $F$-stable.
\end{proposition}

\begin{remark}
$F(t)=\frac{1}{2-m}e^{\left( 2-m\right) t}+C$ ,where $C$ $\geq \frac{1}{m-2}$%
is a constant, fulfills the condition (\ref{11}).
\end{remark}

\subsection{Morse index of the identity map}

Now we are interested by the identity map $I$ on $M$. Let $C$ and $K$ denote
the space of conformal vector fields and the space of Killing vector fields
on $M$ respectively.

\begin{proposition}
\label{prop1} Let $(M,g)$ be a compact $m$-dimensional manifold ( $m\geq 3$%
). Suppose that%
\begin{equation}
\frac{m-2}{m}F^{\prime }\left( \frac{m}{2}\right) -F^{\prime \prime }\left( 
\frac{m}{2}\right) >0  \label{11'}
\end{equation}%
then $Ind_{F}\left( I\right) \geq \dim \left( C/K\right) $.
\end{proposition}

\begin{proof}
Plugging (\ref{7}) in (\ref{8}), we get%
\begin{equation*}
Q_{I}^{F}\left( v\right) =F^{\prime \prime }\left( \frac{m}{2}\right)
\int_{M}div\left( v\right) ^{2}dv_{g}+
\end{equation*}%
\begin{equation}
F^{\prime }\left( \frac{m}{2}\right) \int_{M}\left[ \frac{1}{2}\left\vert
L_{v}g\right\vert ^{2}-div\left( v\right) ^{2}\right] dv_{g}  \label{12}
\end{equation}%
and if $v$ is a conformal vector field on $M$ then ( see the proof of
Theorem2 in \cite{15} )%
\begin{equation}
L_{v}g=-\frac{2}{m}div(v)g  \label{13}
\end{equation}%
where $m=\dim (M)$. So (\ref{12}) becomes%
\begin{equation*}
Q_{I}^{F}\left( v\right) =\left( F^{\prime \prime }\left( \frac{m}{2}\right)
+\frac{2-m}{m}F^{\prime }\left( \frac{m}{2}\right) \right) \int_{M}div\left(
v\right) ^{2}dv_{g}\text{.}
\end{equation*}%
If $\frac{m-2}{m}F^{\prime }\left( \frac{m}{2}\right) -F^{\prime \prime
}\left( \frac{m}{2}\right) >0$, then%
\begin{equation*}
Q_{I}^{F}\left( v\right) \leq 0.
\end{equation*}%
The equality holds if $div(v)=0$ which means by (\ref{13}) that $v$ is a
Killing vector field. Then on the quotient space $C/K$, we have 
\begin{equation*}
Q_{I}^{F}\left( v\right) <0
\end{equation*}%
i.e. 
\begin{equation*}
Ind_{F}\left( I\right) \geq \dim (C/K)\text{.}
\end{equation*}
\end{proof}

\begin{remark}
$F(t)=\frac{m}{m-2}e^{\frac{m-2}{m}t}+Ct$ ,where $C$ $>0$ is a constant,
fulfills the condition (\ref{11'}).
\end{remark}

\subsection{Morse index of the homothetic map}

Let $\phi :\left( M,g\right) \rightarrow \left( N,h\right) $ be a homothetic
map i.e. $\phi ^{\ast }h=k^{2}g$ where $k\in R$. Clearly $\left\vert d\phi
\right\vert _{h}^{2}=mk^{2}$, where $m=\dim (M)$, in that case the $F$%
-tension $\tau _{F}\left( \phi \right) $ is proportional to the mean
curvature of $\phi $ so $\phi $ is $F$-harmonic if and only if $\phi $ is
minimal immersion.

\begin{proposition}
\label{prop2} Let $\phi :\left( M,g\right) \rightarrow \left( N,h\right) $
be an $F$- harmonic homothetic map. Then we have%
\begin{equation*}
Ind_{F}\left( \phi \right) \geq Ind_{F}\left( I\right)
\end{equation*}%
where $I$ is the identiy map of $M$.
\end{proposition}

\begin{proof}
The second variation of $\phi $ in direction of a vector field $v$ reduces to%
\begin{equation*}
Q_{\phi }^{F}\left( v\right) =F"(\frac{mk^{2}}{2})\int_{M}\left\langle
\nabla v,d\phi \right\rangle _{\phi ^{-1}\NEG{T}N}^{2}dv_{g}
\end{equation*}%
\begin{equation}
+F^{\prime }\left( \frac{mk^{2}}{2}\right) \int_{M}\left[ \left\vert \nabla
v\right\vert ^{2}-\sum_{i=1}^{m}\left\langle R^{N}\left( v,d\phi \left(
e_{i}\right) \right) d\phi \left( e_{i}\right) ,v\right\rangle \right] dv_{g}
\label{14}
\end{equation}%
where $\left\{ e_{i}\right\} _{1\leq i\leq m}$ is an orthonormal basis on $M$%
. Let $\Gamma ^{T}\left( \phi \right) $ the subspace of $\Gamma \left( \phi
^{-1}TN\right) $ , consisting of vector fields on $N$ of the form $d\phi
\left( X\right) $ where $X$ is a vector field on $M$. The restriction of $%
Q_{\phi }^{I}$ to $\Gamma ^{T}\left( \phi \right) $, where $I$ is the
identity map on $M$, is given by (see Lemma 2.5 \cite{15} ) 
\begin{equation}
Q_{\phi }^{I}\left( d\phi (X)\right) =k^{2}Q_{I}^{I}\left( X\right) \text{.}
\label{15}
\end{equation}%
As in \cite{15} and since $\nabla d\phi $ takes its value in the normal
fiber bundle of $N$, we get%
\begin{equation*}
\left\langle \nabla _{X}d\phi \left( Y\right) ,d\phi (Z)\right\rangle
=\left\langle \left( \nabla d\phi \right) \left( X,Y\right) ,Z\right\rangle
+\left\langle d\phi \left( \nabla _{X}Y\right) ,d\phi (Z)\right\rangle
\end{equation*}%
\begin{equation}
=k^{2}\left\langle \nabla _{X}Y,Z\right\rangle \text{.}  \label{16}
\end{equation}%
Replacing (\ref{16}) and (\ref{15}) in (\ref{14}) we deduce that%
\begin{equation*}
Q_{\phi }^{F}\left( d\phi (X)\right) =F"(\frac{mk^{2}}{2})k^{2}\int_{M}\left%
\langle \nabla _{e_{i}}X,e_{i}\right\rangle ^{2}dv_{g}+F^{\prime }\left( 
\frac{mk^{2}}{2}\right) k^{2}Q_{I}^{I}(X)
\end{equation*}%
\begin{equation*}
=k^{2}Q_{I}^{F}\left( X\right) \text{.}
\end{equation*}
\end{proof}

Propositions (\ref{prop1}) and (\ref{prop2}) lead to

\begin{corollary}
\label{cor} Let $\phi :\left( M,g\right) \rightarrow \left( N,h\right) $ be
an $F$- harmonic homothetic map. Suppose that Suppose that%
\begin{equation*}
\frac{m-2}{m}F^{\prime }\left( \frac{m}{2}\right) -F^{\prime \prime }\left( 
\frac{m}{2}\right) >0
\end{equation*}%
where $m=\dim (M)$ $\geq 3$.

Then 
\begin{equation*}
Ind_{F}\left( \phi \right) \geq \dim \left( C/K\right) .
\end{equation*}
\end{corollary}

We can deduce an estimation to the $F$-index of an homothetic $F$-harmonic
from Theorem \ref{theo}.

Consider $\phi :\left( M,g\right) \rightarrow \left( S^{n},can\right) $ an
homothetic map i.e. $\phi ^{\ast }can=k^{2}g$, $k\in R$; where $S^{n}$
denotes the unit Euclidean $n$-dimensional sphere endowed with the canonical
metric $can$. The $F$-stress-energy tensor given by (\ref{4'}) writes 
\begin{equation*}
S_{g}^{F}(\phi )=F^{\prime }(\frac{\left\vert d\phi \right\vert ^{2}}{2}%
)\left\vert d\phi \right\vert ^{2}g-2\left( F^{\prime }(\frac{\left\vert
d\phi \right\vert ^{2}}{2})+F^{\prime \prime }\left( \frac{\left\vert d\phi
\right\vert ^{2}}{2}\right) \frac{\left\vert d\phi \right\vert ^{2}}{2}%
\right) \frac{\left\vert d\phi \right\vert ^{2}}{m}g
\end{equation*}%
\begin{equation*}
=\left( \left( 1-\frac{2}{m}\right) F^{\prime }(\frac{\left\vert d\phi
\right\vert ^{2}}{2})-\frac{\left\vert d\phi \right\vert ^{2}}{m}F^{\prime
\prime }\left( \frac{\left\vert d\phi \right\vert ^{2}}{2}\right) \right)
\left\vert d\phi \right\vert ^{2}g\text{.}
\end{equation*}%
So $S_{g}^{F}(\phi )$ will be positive defined if $\left( 1-\frac{2}{m}%
\right) F^{\prime }(\frac{\left\vert d\phi \right\vert ^{2}}{2})-\frac{%
\left\vert d\phi \right\vert ^{2}}{m}F^{\prime \prime }\left( \frac{%
\left\vert d\phi \right\vert ^{2}}{2}\right) >0$. As a consequence of
Theorem \ref{theo}, we have

\begin{proposition}
\label{prop3} Let $\phi $ be an homothetic $F$-harmonic map from a compact $%
m-$Riemannian manifold $\left( M,g\right) $ ($m\geq 3$) into the Euclidean
sphere $S^{n}$. Suppose that 
\begin{equation}
\left( 1-\frac{2}{m}\right) F^{\prime }(\frac{\left\vert d\phi \right\vert
^{2}}{2})-\frac{\left\vert d\phi \right\vert ^{2}}{m}F^{\prime \prime
}\left( \frac{\left\vert d\phi \right\vert ^{2}}{2}\right) >0\text{.}
\label{17}
\end{equation}%
Then the Morse index of $\phi $, $Ind_{F}(\phi )\geq n+1$.
\end{proposition}

\begin{remark}
The function $F(t)=\frac{m^{2}}{m-2}e^{\frac{m-2}{m^{2}}t}$, with $m\geq 3$
fulfills the condition (\ref{17}) for homothetic maps $\phi
:(M,g)\rightarrow (S^{n},can)$ i.e. $\phi ^{\ast }can=k^{2}g$ provided that $%
k^{2}<m$.
\end{remark}

\begin{remark}
The space $C$ of conformal vector fields on the unit Euclidean sphere $S^{n}$
is of dimension $\frac{1}{2}\left( n+1\right) \left( n+2\right) $\ and that
of Killing vector fields $K$ is of dimension $\frac{1}{2}n\left( n+1\right) $%
. Then $\dim (C/K)=n+1$. So we recover the result given by Corollary \ref%
{cor}.
\end{remark}

\end{document}